\documentclass[12pt,sumlimits,intlimits]{amsart}
\usepackage{url,cite,hyperref,amsmath,amsthm,amssymb,mathtools}
\usepackage[margin=1.25in]{geometry}

\newtheorem{theorem}{Theorem}[section]

\newtheorem{corollary}[theorem]{Corollary}
\newtheorem{lemma}[theorem]{Lemma}
\newtheorem{question}[theorem]{Question}
\newtheorem{definition}[theorem]{Definition}
\newtheorem{remark}[theorem]{Remark}

\newtheorem*{thm16}{Theorem 1.6}
\newcommand{\la}{\langle}
\newcommand{\ra}{\rangle}

\newcommand{\V}{\Vert}

\newcommand{\Z}{\mathbb{Z}}

\newcommand{\N}{\mathbb{N}}

\newcommand{\mf}{\mathcal{F}}

\newcommand{\mh}{\mathcal{H}}
\newcommand{\pr}{\textrm{Prim}(G)}
\newcommand{\ve}{\varepsilon}

\numberwithin{equation}{subsection}

\title{Quasidiagonal Representations of Nilpotent Groups}
\author{Caleb Eckhardt}
\address{Department of Mathematics, Miami University, Oxford, OH, 45056}
\email{eckharc@miamioh.edu}
\thanks{Partially supported by NSF grant DMS-1262106.}

\date{}
\begin{document}
\maketitle
\begin{abstract} We show that every unitary representation of a discrete solvable virtually nilpotent group $G$ is quasidiagonal.  Roughly speaking, this says that every unitary representation of $G$ approximately decomposes as a direct sum of finite dimensional approximate representations. In operator algebraic terms we show that $C^*(G)$ is strongly quasidiagonal.
\end{abstract}
\section{Introduction}  Murray and von Neumann cite the study of unitary group representations as one of four key motivations for their development of operator algebra theory \cite{Murray36}.  The last seventy-five years witnessed numerous intimate interactions between the theories,  completely validating their motivation.  The goal of this paper is to obtain yet another connection between representation theory and operator algebras. On one hand we use a natural approximation property of C*-algebras to obtain information about unitary representations of discrete nilpotent groups. On the other hand we employ classic results about nilpotent groups to produce some new examples of strongly quasidiagonal C*-algebras, and simple, nuclear quasidiagonal C*-algebras.

A  linear operator on a Hilbert space is called \emph{quasidiagonal} if it is a compact perturbation of a direct sum of finite rank operators.  One analogously defines quasidiagonality of a set of operators, and hence of a representation of a C*-algebra (see Definition \ref{def:QD}). 
Interpreting quasidiagonality locally for a unitary group representation translates to declaring a unitary representation $\pi:G\rightarrow B(\mh)$  quasidiagonal if for every finite subset $\mf$ of $G$ and $\ve>0$, there are mutually orthogonal, finite rank projections $Q_n\in B(\mh)$ such that 
\begin{equation*}
\max_{t\in\mf}\Big\V \bigoplus_{n}Q_n\pi(t)Q_n-\pi(t)  \Big\V<\ve.
\end{equation*}
Note that this implies that the function $t\mapsto Q_n\pi(t)Q_n$ is almost multiplicative on $\mf$ and that $Q_n\pi(t)Q_n\in B(Q_n(\mh))$ is almost unitary.  In other words, $\pi$ is quasidiagonal if it locally approximately decomposes as a direct sum of finite dimensional approximate unitary representations.

Rosenberg proved \cite{Hadwin87} that the left regular representation of a non amenable discrete group is not quasidiagonal (see also \cite{Carrion13} for a quantitative version of this theorem). It is a long-standing open question whether or not the left regular representation of every amenable group is quasidiagonal (see \cite{Carrion13} for recent progress and the state of the art).  Following Hadwin, we call a group \emph{strongly quasidiagonal} if every unitary representation is quasidiagonal.

There are many examples in \cite{Carrion13} of amenable groups whose left regular representation is quasdiagonal, while the group is not strongly quasidiagonal.  The commonality between all of the examples is exponential growth; indeed it is precisely the growth conditions that lead to the non quasidiagonal representations.
On the other hand, it is fairly straightforward to see that every representation of an abelian group is quasidiagonal.  Since nilpotent groups (Definition \ref{def:nil}) posses a large degree of commutativity and have polynomial growth, it was natural to consider the problem of whether or not every nilpotent group is strongly quasidiagonal.  Moreover there are  representation theoretic simplifications present in nilpotent groups that suggest strong quasidiagonality. We recall the relevant facts.

Due to the fact that a discrete group is Type I if and only if it is virtually abelian \cite{Thoma64} (a group is virtually ``P" if it has a finite index ``P" subgroup), to study representations of nilpotent groups one usually looks for replacements for the dual. There are two natural candidates for this replacement: The primitive ideal space of $C^*(G)$, and the space of characters on $G.$  

An ideal of the group C*-algebra $C^*(G)$  is called \emph{primitive} if it is the kernel of some irreducible representation of $C^*(G).$  The primitive ideal space of $C^*(G)$, denoted $\pr$, is equipped with the hull-kernel topology.  In general, $\pr$ is topologically poorly behaved, but
Moore and Rosenberg proved \cite{Moore76} that if $G$ is nilpotent and finitely generated, then $\textrm{Prim}(G)$ is $T_1$ (i.e., all of the singleton sets are closed). 
Poguntke later generalized their result with:
\begin{theorem}[Moore \& Rosenberg, Poguntke \cite{Poguntke81}] \label{thm:pog} Let $G$ be a discrete, virtually nilpotent group. Then $G$ has a $T_1$ primitive ideal space. That is, every primitive ideal of $C^*(G)$ is a maximal ideal.
\end{theorem}
This in tandem with Voiculescu's theorem \cite{Voiculescu76} (see also \cite{Davidson96}), significantly simplifies the task of proving that $C^*(G)$ is strongly quasidiagonal.  We use Voiculescu's theorems  often enough in this paper that it is worthwhile recalling them.

In all that follows, let $A$ be a separable C*-algebra.  Recall that two representations $(\rho_1,\mh_1),(\rho_2,\mh_2)$ of $A$ are \emph{approximately unitarily equivalent} if there is a sequence of unitaries $u_n:\mh_1\rightarrow \mh_2$ such that 
\begin{equation*}
\lim_{n\rightarrow\infty}\V \rho_1(a)-u_n^*\rho_2(a)u_n  \V=0 \quad \textrm{ for all }a\in A.
\end{equation*}

\begin{theorem}[Voiculescu's Theorem \cite{Voiculescu76,Davidson96}] \label{thm:voi1}Suppose that $\rho_1$ and $\rho_2$ are representations of $A$ such that $\textrm{ker}(\rho_1)=\textrm{ker}(\rho_2).$ Suppose that neither $\rho_1(A)$ nor $\rho_2(A)$ contain a non-zero compact operator.  Then $\rho_1$ and $\rho_2$ are approximately unitarily equivalent. In particular, any two representations of a simple, unital, infinite dimensional C*-algebra are approximately unitarily equivalent.
\end{theorem}
The following corollary shows why Voiculescu's theorem also goes by the name ``Noncommutative Weyl-von Neumann-Berg Theorem:"
\begin{theorem}[Voiculescu's Theorem \cite{Voiculescu76,Davidson96}] \label{thm:voi2} Let $\pi$ be a representation of $A.$  Then $\pi$ is approximately unitarily equivalent to a direct sum of irreducible representations.
\end{theorem}
The following lemma is well-known and follows more or less from Definition \ref{def:QD}:
\begin{lemma} \label{lem:appuqd} If $\rho_1$ and $\rho_2$ are approximately unitarily equivalent representations and $\rho_1$ is quasidiagonal, then  $\rho_2$ is also quasidiagonal.
\end{lemma}
We call a C*-algebra \emph{quasidiagonal} if it admits a faithful quasidiagonal representation (Definition \ref{def:QD}). 
 \begin{lemma} \label{lem:Voic} Let $A$ be a C*-algebra with $T_1$ primitive ideal space. Then $A$ is strongly quasidiagonal if and only if every primitive quotient of $A$ is quasidiagonal.
\end{lemma}
\begin{proof} By Theorem \ref{thm:voi2} and Lemma \ref{lem:appuqd}  we need only check that irreducible representations of $A$ are quasidiagonal. Let $\pi$ be an irreducible representation of $A.$  Then the kernel of $\pi$, call it $J$, is maximal.   By assumption, $A/J$ is a quasidiagonal C*-algebra.  Since $A/J$  is simple, by Theorem \ref{thm:voi1}, $A/J$ is strongly quasidiagonal.  Since $\pi$ induces a representation $\widetilde{\pi}$ of $A/J$ such that $\pi(A)=\widetilde{\pi}(A/J)$ we are done.
\end{proof}
\begin{remark}  Lemma \textup{\ref{lem:Voic}} does not hold for C*-algebras without a $T_1$ primitive ideal space, even if the algebra has a ``tractable" representation theory.  Larry Brown constructed an interesting example in \textup{\cite{Brown84}} of a primitive, Type I C*-algebra such that every primitive quotient is quasidiagonal, but which also admits a faithful, non quasidiagonal representation. See \textup{\cite{Brown96}} for generalizations.
\end{remark}
As mentioned above, the space of characters of $G$ (see Definiton \ref{def:char}) also serves as a replacement for the dual.   For virtually nilpotent groups there is a nice connection between primitive ideals and characters:
\begin{theorem}[Howe \textup{\cite{Howe77}}, Carey \& Moran \textup{\cite[Theorem 2.1]{Carey84}}] \label{thm:CM} Let $G$ be virtually nilpotent.  An ideal $J$ of $C^*(G)$ is primitive if and only if there is a character $\phi$ on $G$ such that
\begin{equation*}
J=\{ x\in C^*(G): \phi(x^*x)=0 \}.
\end{equation*}
\end{theorem}
\begin{remark} A lot of interesting work has been devoted to deciding precisely the relationship between the primitive ideal space of a nilpotent group and its space of characters.  This correspondence doesn't impact the present work but may be of interest to the reader; we refer those interested to  Kaniuth's article \textup{\cite{Kaniuth06}} and the references therein.
\end{remark}
In summation, to prove that \emph{every} representation of a virtually nilpotent group $G$ is quasidiagonal is now reduced to proving that for every maximal ideal $J$ of $C^*(G)$, there is \emph{some} quasidiagonal representation of $C^*(G)/J.$  Moreover, we have a natural choice for this representation: the GNS representation associated with the character $\phi$ that induces $J.$ This  is precisely the route taken in this paper.

In Section \ref{sec:SQUG} we show that every faithful character on a unitriangular group gives rise to a quasidiagonal representation. We achieve by first proving quasidiagonality of those representations induced from the center and then reduce the general case to the centrally induced case.
The rest of the section is devoted to proving one can reduce the case of a finitely generated, torsion free group with infinite non-central conjugacy classes to the unitriangular groups. Finally, in Section \ref{sec:RTFC} we reduce the general solvable, virtually nilpotent case to the aforementioned classes of groups and characters to prove the following:
\begin{theorem} \label{thm:fgnil} Let $G$ be a finitely generated,  solvable virtually nilpotent group.  Then $C^*(G)$ is strongly quasidiagonal.
\end{theorem}
\begin{remark} The adjective ``solvable" is Theorem \textup{\ref{thm:fgnil}} is unfortunate because it is almost certainly unnecessary.  We were unable to extend the proof of Lemma \textup{\ref{thm:find}} to cover all virtually nilpotent groups, but since all nilpotent groups are (trivially) solvable,  Theorem \textup{\ref{thm:fgnil}} does cover all nilpotent groups.
\end{remark}
As strong quasidiagonality is preserved by taking inductive limits we immediately obtain the following:
\begin{corollary} \label{cor:genqd} Let $G$ be a countable group that is the inductive limit of solvable virtually nilpotent groups, then $C^*(G)$ is strongly quasidiagonal. In particular $C^*(G)$ is strongly quasidiagonal whenever $G$ is a discrete nilpotent group.
\end{corollary}
Let $A_\infty$ be the group of even permutations of $\N$ with finite support.  Then $A_\infty$ is a simple group that is an inductive limit of finite groups.  Therefore $C^*(A_\infty)$ is an inductive limit of finite dimensional C*-algebras, so it is strongly quasidiagonal.  More generally, by exploiting the fact that virtual nilpotency is not preserved by inductive limits while strong quasidiagonality is preserved by inductive limits one can easily obtain examples of non virtually nilpotent groups that are strongly quasidiagonal. For finitely generated groups, Theorem \ref{thm:fgnil}  raises the possibility of detecting virtual nilpotence  C*-algebraically, leading to the following 
\begin{question} \label{ques:fgsqd} Are there any \emph{finitely generated} strongly quasidiagonal groups that are not virtually nilpotent? [See note added in proof]
\end{question}
In the non virtually nilpotent case, a new difficulty appears: Moore and Rosenberg showed \cite{Moore76} that if a group is not virtually nilpotent, it cannot have a $T_1$ primitive ideal space.  
\subsection{Definitions and Notation}
We refer the reader to Brown \& Ozawa \cite{Brown08} for information on amenability, group C*-algebras and quasidiagonality (and any other operator algebraic topic covered here)   and to  P. Hall's lecture notes \cite{Hall88} for information about nilpotent groups.
\begin{definition} \label{def:QD} Let $\mh$ be a separable Hilbert space and $\mathcal{S}\subseteq B(\mh)$ be a subset of the bounded linear operators on $\mh.$ We say that $\mathcal{S}$ is a \textbf{quasidiagonal} (henceforth QD) set of operators if there is a sequence of finite rank orthogonal projections $P_n\in B(\mh)$ such that $P_n(\xi)\rightarrow \xi$ as $n\rightarrow\infty$ for all $\xi\in\mh$ and
\begin{equation*}
\lim_{n\rightarrow\infty}\V P_nT-TP_n  \V=0, \quad \textrm{ for all }\quad T\in \mathcal{S}.
\end{equation*}
A representation $(\pi,\mh)$ of a C*-algebra $A$ is QD if $\pi(A)$ is a QD set of operators.  A C*-algebra is QD if it admits a faithful QD representation.  A C*-algebra is strongly QD \textup{\cite{Hadwin87}} if every representation is QD. 
\end{definition}
\begin{definition} \label{def:nil}
Let $G$ be a group. We denote by $e$ the neutral element of $G.$  We denote the center of $G$ as $Z(G).$ One defines $Z_0(G)=\{ e \}$ and inductively defines $Z_n(G)$ so that $Z(G/Z_{n-1}(G))=Z_n(G)/Z_{n-1}(G).$ We say that $G$ is \textbf{nilpotent} if $Z_c(G)=G$ for some $c\in\N.$ We say $G$ is \textbf{virtually nilpotent} if there is a finite index subgroup $N\leq G$ such that $N$ is nilpotent.  

A group $G$ is \textbf{solvable} if there is a series $\{ e \}=G_{n+1}\trianglelefteq G_n\trianglelefteq \cdots G_2\trianglelefteq G_1=G$ such that $G_i/G_{i+1}$ is abelian. Every nilpotent group is solvable and the only fact about solvable groups we require here is that if $G$ is solvable and non-trivial, then the quotient of $G$ by its commutator subgroup is abelian and non-trivial. It is clear from the definitions that every solvable group is amenable.
\end{definition}
\begin{definition} \label{def:char} Let $\phi$ be a positive definite function on $G$ with $\phi(e)=1.$  We write $(\pi_\phi,L^2(G,\phi))$ for the associated Gelfand-Naimark-Segal (GNS) representation of $C^*(G).$ We say $\phi$ is a \textbf{trace} if $\phi(g^{-1}xg)=\phi(x)$ for all $x,g\in G.$ The set of traces form a weak-* compact, convex subset of $C^*(G)^*.$  We call the extreme points of the trace space \textbf{characters}.  A character $\phi$ is \textbf{faithful} if $\phi(g)=1$ implies $g=e.$ Every character gives rise to a factor representation of $G$ (see \textup{\cite[Theorem 6.7.3]{Dixmier77}}), therefore every character  is multiplicative on $Z(G)$.
\end{definition}
\section{Quasidiagonality of Faithful Representations} \label{sec:SQUG}

In this section we show that characters $\phi$ on unitriangular groups $U_d$ (see Section \ref{sec:unitriangular} for definitions),  that vanish off of their center $Z_d$,   give rise to QD representations.  For $x\in U_d$, the idea is to view $\pi_\phi(x)$ (see Definition \ref{def:char}) as a weighted permutation operator on $B(\ell^2(U_d/Z_d)).$ Since $U_d/Z_d$ is residually finite, we then employ Orfanos's work \cite{Orfanos11} on quasidiagonality of residually finite groups to prove the permutation part of $\pi_\phi(x)$ is QD and then  use specific properties of $U_d$ to show that the weights are ``locally approximately constant on cosets" (see Lemma \ref{lem:diagqd}) to obtain QD of the representation $\pi_\phi.$

\subsection{Orfanos's Projections} \label{sec:Orfanos} We quickly recall Orfanos's algorithm \cite{Orfanos11} for generating QD sequences for the left regular representation of a residually finite amenable group. 

 Let $G$ be an amenable residually finite group and let $(L_n)$ be a sequence of finite index normal subgroups with trivial intersection.  Let $K_n$ be a (finite) set of $G/L_n$ coset representatives such that there is an exhaustive F\o lner sequence $F_n\subseteq K_n$ (i.e. for every $x\in G$ we have $| F_n \Delta F_nx ||F_n|^{-1}\rightarrow0$ as $n\rightarrow\infty$ and $\cup F_n=G.$ )

One then defines the functions $\phi_n:G\rightarrow [0,1]$ by
\begin{equation*}
\phi_n(x)=\sqrt{\frac{|K_n\cap F_nx|}{|F_n|}}.
\end{equation*} 
Notice that 
\begin{equation}
\textrm{supp}(\phi_n)=F_n^{-1}K_n. \label{eq:supp}
\end{equation}
Then for each $y\in K_n$, one defines the norm one vector $\xi_{yL_n}=\sum_{\alpha\in yL_n}\phi_n(\alpha)\delta_\alpha\in \ell^2G$ and $P_n\in B(\ell^2G)$ as the (finite rank) projection onto $\textrm{span}\{ \xi_{yL_n}: y\in K_n \}.$ Let $\lambda:G\rightarrow B(\ell^2G)$ be the left regular representation. 
\begin{theorem}[Orfanos \textup{\cite[Lemma 1.7]{Orfanos11}}] \label{thm:stefanos} For every $s\in G$, we have  
\begin{equation*}
 \lim_{n\rightarrow\infty}\V \lambda(s)P_n-P_n\lambda(s)  \V= 0, 
\end{equation*} 
and $P_n$ converges pointwise to the identity.
\end{theorem}

\subsection{Quasidiagonality of Special Representations of $U_d$} \label{sec:unitriangular}  
Let $d\geq1$ and let $U_d\leq GL_d(\Z)$ be the unitriangular group of $d\times d$ matrices; that is $U_d$ are those upper triangular matrices with 1's along the diagonal. It is fairly straightforward to verify that each $U_d$ is nilpotent. We freely use the fact that $U_d$ is contained in the ring $M_d(\Z)$ and $\pm$ will always refer to ring operations and $1\in U_d$ the identity matrix.  In particular, notice that $U_d$ is closed under the operation $x -y+1$, where $x,y\in U_d$.

For an element $a\in M_d(\Z)$, we write $a_{ij}\in \Z$ for the $(i,j)$ matrix coefficient of $a$ and we write $e_{ij}$ for the matrix with value $1$ in the $(i,j)$ coordinate and 0 elsewhere (for example $e_{12}\not\in U_d$ while $1+e_{12}\in U_d$). Let $Z_d$ be the center of $U_d.$ Then 
\begin{equation*}
Z_d=\{ 1+ne_{1d}:n\in\Z \}\cong \Z.
\end{equation*} 
We now make the ideas of Section \ref{sec:Orfanos} specific to our situation. In order to reduce the number of subscripts by one, we define $G=U_d$ and $Z=Z_d.$
Throughout this section we fix a character  $\phi$  such that $\phi(x)=0$ if $x\not\in Z$  and for definiteness suppose $\phi$ on $Z\cong\Z$ is defined by $\phi(n)=e^{2\pi i\theta n}$ for some fixed $\theta\in [0,1).$  In this section we prove that the GNS representation associated with $\phi$ is QD (Theorem \ref{thm:QDtrace})
\\\\
As in Definition \ref{def:char} we denote by $L^2(G,\phi)$ the GNS Hilbert space associated with $\phi$ and $\pi_\phi:G\rightarrow B(L^2(G,\phi))$ the GNS representation.
\begin{definition} \label{def:coset}
We define a set of  coset representatives of $G/Z$ as
\begin{equation*}
C=\{ x\in G:x_{1d}=0  \}.
\end{equation*}
For each $x\in G/Z $ we write $\widetilde{x}$ for the unique element of $C$ such that $\widetilde{x}Z=x.$
\end{definition}
\begin{definition} For $a\in G$, let $\delta_a\in L^2(G,\phi)$ be the canonical image of $a$ and let $\delta_{aZ}\in \ell^2(G/Z)$ be the canonical image of $aZ.$
\end{definition}
\begin{lemma}  \label{lem:ONB}  Let $a\in G$ and let $a'=a(1-a_{1d}e_{1d})\in C.$  Then 
\begin{equation}
\delta_a=\phi(1+a_{1d}e_{1d})\delta_{a'}=e^{2\pi i\theta a_{1d}}\delta_{a'}. \label{eq:proportion}
\end{equation}
Hence the set $X=\{ \delta_x:x\in C \}$ is an orthonormal basis for $L^2(G,\phi).$
\end{lemma}
\begin{proof} If $x\neq y$ with $x,y\in C$, then $y^{-1}x\not\in Z$, hence $\phi(y^{-1}x)=0$, proving that $X$ is an orthonormal set.  Let $a\in G.$  Then  $\la  \delta_a,\delta_{a'} \ra_\phi=\phi(1+a_{1d}e_{1d})=e^{2\pi ia_{1d}\theta}.$  Since both $\delta_a$ and $\delta_{a'}$ are norm 1, they are proportional, in fact we must have $\delta_a=e^{2\pi ia_{1d}\theta}\delta_{a'}.$  This proves (\ref{eq:proportion}) and shows that the closure of the span of $X$ is  $L^2(G,\phi).$
\end{proof}

Let $W: L^2(G,\phi)\rightarrow \ell^2(G/Z)$ be defined by $W(\delta_{x})=\delta_{xZ}$ where $x\in C.$ Then $W$ is unitary by Lemma \ref{lem:ONB}. Let $y\in G$ and $x\in C.$  Then by Lemma \ref{lem:ONB}, we have
\begin{align*}
W\pi_\phi(y)W^*(\delta_{xZ})&=W\pi_\phi(y)\delta_x\\
&= \phi\Big(1+\sum_{i=2}^d y_{1i}x_{id}e_{1d}\Big)\delta_{yxZ}.
\end{align*}
For each $y\in G$ and $x\in C$ define 
\begin{equation} \label{eq:psidef}
\psi_y(x)=\phi\Big(1+\sum_{i=2}^d y_{1i}x_{id} e_{1d}\Big)=\exp\Big(    2\pi i\theta\sum_{i=2}^d y_{1i}x_{id}   \Big),
\end{equation}
and $D_y\in B(\ell^2(G/Z))$ by
\begin{equation}
D_y(\delta_{xZ})=\psi_y(x)\delta_{xZ}.
\end{equation}
Then $W\pi_\phi(y)W^*=\lambda_{G/Z}(y)D_y,$ where $\lambda_{G/Z}:G\rightarrow B(\ell^2(G/Z))$ is the left regular representation.  Summarizing the above discussion produces 
\begin{lemma} \label{lem:reduce}
To prove that $\pi_\phi$ is a QD representation, it suffices to prove that 
\newline
$\{ \lambda_{G/Z}(y), D_y:y\in G  \}\subseteq B(\ell^2(G/Z))$ is a QD set of operators.
\end{lemma}
\begin{remark} The benefit of twisting the representation $\pi_\phi$ over to $\ell^2(G/Z)$ is that it allows us to use the results of Section \textup{\ref{sec:Orfanos}} without modification.
\end{remark}
\begin{definition} Let $C$ be as in Definition \textup{\ref{def:coset}}. Let $a,b\in \mathbb{R}.$  Define
\begin{equation*}
G[a,b]=\{  x\in C: x_{ij}\in [a,b] \textrm{ if }i<j  \}\subseteq C\subseteq G.
\end{equation*}
Fix $n\in\mathbb{N}.$ Define the normal subgroup $\widetilde{L}_n\trianglelefteq G$ as
\begin{equation*}
\widetilde{L}_n=\{  x\in G: x_{ij}\in n\Z \quad \textrm{for }i<j \},
\end{equation*}
and the set:
\begin{equation*}
\widetilde{K}_n=\left\{ \begin{array}{ll} G[-(n-1)/2,(n-1)/2] & \textrm{ if }n\textrm{ is odd}\\
                                                            G[-n/2+1,n/2] & \textrm{ if }n\textrm{ is even}
                                                            \end{array} \right..
\end{equation*}
Let $\pi:G\rightarrow G/Z$ be the quotient homomorphism.  We define $L_n=\pi(\widetilde{L}_n)$ and $K_n=\pi(\widetilde{K}_n).$  One then easily checks that $K_n$ is a set of coset representatives of $\pi(G)/L_n$. The sequence $K_n$  is exhaustive so it contains a F\o lner sequence. Let us now fix a sequence of finite sets $F_n\subseteq G/Z$ with the properties:
\begin{enumerate}
\item $F_n\subseteq K_n$
\item $\cup F_n=G/Z$
\item $(F_n)$ is a F\o lner sequence.
\item $\pi^{-1}(F_n^{-1})\cap C\subseteq G[-n^{1/4},n^{1/4}].$
\end{enumerate}
\end{definition}
We now show that the weights defined in (\ref{eq:psidef}) are locally almost constant on $L_n$-cosets.
\begin{lemma} \label{lem:diagqd}
For every $\ve>0$ and finite subset $\mf\subseteq G$, there is an $n\in\N$ such that 
\begin{equation*}
\max_{z\in\mf}\max_{y\in K_n}\max_{\alpha\in yL_n\cap F_n^{-1}K_n}|\psi_z(\widetilde{\alpha})-\psi_z(\widetilde{y})|<\ve.
\end{equation*}
\end{lemma}
\begin{proof} Let $p$ be the maximum value of the absolute value of all the entries of the matrices in $\mathcal{F}.$  Choose $n>p^4$ such that $\textrm{dist}(\theta n,\Z)<\frac{1}{n}$ and  $2\pi d^2n^{-1/2}<\ve$  ( for example, take $n$ to be a large enough denominator of  one of the convergents of $\theta$ in its continued fraction expansion,  see \cite{Leveque96}).

Fix an element $y\in K_n$ and let $\widetilde{y}\in \widetilde{K}_n$ be its lift. Let $\alpha\in yL_n\cap F_n^{-1}K_n.$  Set $y^\alpha=\widetilde{\alpha}-\widetilde{y}+1$. Then elementary calculations show that $y^\alpha\in G[-dn^{5/4},dn^{5/4}]\cap \widetilde{L}_n.$    
Let $z\in\mf.$ Then
\begin{align*}
\psi_z(\widetilde{\alpha})-\psi_z(\widetilde{y})&=\exp\Big(2\pi i\theta  \sum_{k=2}^d z_{1k}\widetilde{\alpha}_{kd} \Big)-\exp\Big(2\pi i\theta  \sum_{k=2}^d z_{1k}\widetilde{y}_{kd} \Big)\\
&=\exp(2\pi i\theta z_{1d})\Big[\exp\Big(2\pi i\theta  \sum_{k=2}^{d-1} z_{1k}\widetilde{\alpha}_{kd} \Big)-\exp\Big(2\pi i\theta  \sum_{k=2}^{d-1} z_{1k}\widetilde{y}_{kd} \Big)\Big]\\
&=\exp(2\pi i\theta z_{1d})\exp\Big(2\pi i\theta\sum_{k=2}^{d-1} z_{1k}\widetilde{y}_{kd} \Big)\Big[\exp\Big(2\pi i\theta  \sum_{k=2}^{d-1} z_{1k}y^\alpha_{kd} \Big)-1\Big].
\end{align*}
Define $z'=z(1-z_{1d}e_{1d})\in C.$ From the above identity we immediately obtain
\begin{equation*}
|\psi_z(\widetilde{\alpha})-\psi_z(\widetilde{y})|=|\psi_{z'}(y^\alpha)-1|. 
\end{equation*}
We have
\begin{equation*}
\psi_{z'}(y^\alpha)=\exp\Big(2\pi i\theta\sum_{k=2}^{d-1}z_{1k}y^\alpha_{kd}\Big)=\prod_{k=2}^{d-1}\exp(2\pi i\theta z_{1k}y^\alpha_{kd}).
\end{equation*}
Since $z\in\mf$ we have $|z_{1k}|\leq n^{1/4}$ for $1\leq k \leq d.$ Since $y^\alpha\in G[-dn^{5/4},dn^{5/4}]\cap \widetilde{L}_n$, for $2\leq k\leq d-1$, we have   $z_{1k}y^\alpha_{kd}=nm_k$ for some integer $m_k$ with $|m_k|\leq dn^{(1/4+1/4)}$, for each $2\leq k\leq d-1.$ Since $\textrm{dist}(\theta n,\Z)<n^{-1}$, we have $\textrm{dist}(\theta z_{1k}y^\alpha_{kd},\Z)<dn^{-1/2}$ for all $2\leq k\leq d-1.$ Hence $|\psi_{z'}(y^\alpha)-1|<2\pi d^2n^{-1/2}<\ve$. 
\end{proof}
\begin{theorem} \label{thm:QDtrace} Let $\phi$ be a character of $G=U_d$ that vanishes off $Z=Z(G)$, then $\pi_\phi$ is a QD representation.
\end{theorem}
 \begin{proof}  By Lemma \ref{lem:reduce}, it suffices to prove that $\{ \lambda_{G/Z}(y), D_y:y\in G  \}\subseteq B(\ell^2(G/Z))$ is a QD set of operators.
 
Using our specific subsets $K_n$ and $L_n$ of $G/Z$, form the functions $\phi_n:G/Z\rightarrow [0,1]$, vectors $\xi_{yL_n}\in \ell^2(G/Z)$ and finite rank projections $P_n\in B(\ell^2(G/Z))$ as in Section \ref{sec:Orfanos}. By \cite{Orfanos11} (Theorem \ref{thm:stefanos} above), we have 
\begin{equation}
\V \lambda_{G/Z}(y)P_n-P_n\lambda_{G/Z}(y) \V\rightarrow 0 \quad \textrm{as }n\rightarrow\infty\quad\textrm{  for all }y\in G.   \label{eq:ltcom}
\end{equation}
Let $\mf\subseteq G$ be finite and $\ve>0.$  Choose  $n\in\N$ for the pair $(\mf,\ve)$  that satisfies Lemma \ref{lem:diagqd}. Let $z\in\mf.$ Then for each $y\in K_n$, we have
\begin{align*}
\V D_z(\xi_{yL_n})-\psi_z(\widetilde{y})\xi_{yL_n}  \V^2&= \sum_{\alpha\in yL_n}|\phi_n(\alpha)|^2|\psi_z(\widetilde{\alpha})-\psi_z(\widetilde{y})|^2\\
&=\sum_{\alpha\in yL_n\cap F_n^{-1}K_n}|\phi_n(\alpha)|^2||\psi_z(\widetilde{\alpha})-\psi_z(\widetilde{y})|^2 \quad \textrm{by (\ref{eq:supp})}\\
&\leq \ve^2,  
\end{align*}
where the last line follows by Lemma \ref{lem:diagqd}. Notice that  $D_z(\xi_{yL_n})$ is perpendicular to $\xi_{y'L_n}$ when $y\neq y'.$ This observation combined with the above estimates shows that $\V D_zP_n-P_nD_z \V\leq\ve,$ for every $z\in\mf.$  It now follows that there is a subsequence $n_k$  such that 
\begin{equation*}
\lim_{k\rightarrow\infty}\V D_zP_{n_k}-P_{n_k}D_z\V=0\quad \textrm{ for all } \quad z\in G.
\end{equation*}
This coupled with (\ref{eq:ltcom}) proves the Theorem.
\end{proof}

\subsection{Dimension Subgroups} \label{sec:dimsub}
It is well-known that any finitely generated, torsion-free nilpotent group can be embedded into $U_d.$ (Swan has provided a very short proof of this fact \cite{Swan67}).  In  \cite{Jennings55}, Jennings provides an explicit embedding, and  by following through his proof one realizes that the slightest of modifications provides the following 
\begin{lemma}[\textup{\cite{Jennings55}}] \label{lem:correctplace} Let $G$ be a finitely generated, torsion free nilpotent group, and let $a_1,...,a_d$ be free generators of $Z(G).$  Then there are unitriangular groups $U_{k_1},...,U_{k_d}$ and homomorphisms $\rho_i:G\rightarrow U_{k_i}$ such that $\rho_i(a_i)$ generates $Z(U_{k_i})$, $a_i\in \textup{ker}(\rho_j)$ for $i\neq j$ and the direct sum $\rho=\rho_1\times\cdots\times \rho_d$ is faithful on $G.$
\end{lemma}
\begin{proof} Most of the proof is simply recalling results from \cite{Jennings55} which we shall do for the convenience of the reader.   We follow a mix of P. Hall's lecture notes \cite[Section 7]{Hall88} and Hartley's conference proceedings \cite{Hartley82} for the results on dimension subgroups. We remark that although Hall's discussion focuses on dimension subgroups relative to a field of characteristic zero, everything we use below works equally well for dimension subgroups relative to $\Z$ (see e.g. \cite{Hartley82})

Set $N_1=\la a_2,...,a_d  \ra\trianglelefteq G.$   By \cite{Malcev49} (see also \cite[Theorem 1.2]{Jennings55}), $G/Z(G)$ is torsion free, from which it immediately follows that $G_1:=G/N_1$ is also torsion free. Let $b_1$ be the image of $a_1$ in $G_1.$

Form the group ring $\Z[G_1].$  The \emph{augmentation ideal}, denoted $\Delta$ is the two-sided ideal generated by elements of the form $1-g$ for $g\in G_1.$ For each $n\geq 1$, the \emph{dimension subgroup} of $G_1$ is defined as 
$\delta_n(G_1)=G_1\cap (1+\Delta^n).$ Then $\delta_n(G_1)$ forms a terminating central series \cite{Hartley82}. Let $c$ be such that 
\begin{equation}
b_1\in \delta_c(G_1)\setminus \delta_{c+1}(G_1).
\end{equation}
 Let $\rho_1:G_1\rightarrow End(\Z[G_1]/\Delta^{c+1})$ be defined by left multiplication.  It follows by the discussion preceding  \cite[Theorem 7.5]{Hall88} that (since $b_1$ is a free generator of $Z(G_1)$) there is an ordered basis for the free abelian group $\Z[G_1]/\Delta^{c+1}$, given by 
\begin{equation*}
\beta=\{ 1-b_1, x_{c,1},...,x_{c,k_c},..., x_{1,1},...,x_{1,k_1}, 1  \} \textrm{ mod } \Delta^{c+1}
\end{equation*}
 where $x_{i,j}\in \Delta^i\setminus \Delta^{i+1}$ (notice that the weight of $1-b_1$ equals $c$ in Hall's notation).  
 
 With respect to the basis $\beta$ it is obvious that the matrix of $1-\rho_1(b)$ is upper triangular with 0s along the diagonal for all $b\in G_1$ and that $1-\rho_1(b_1)$ is the rank one operator that maps $1$ to $(1-b_1)$, i.e. that $\rho_1(b_1)$ generates the center of the upper triangular matrices with respect to $\beta.$
 Similarly one defines $N_2,...,N_d$ and representations $\rho_2,...,\rho_d$ as above. Let $\pi_i:G\rightarrow G_i$ be the quotient map and define $\rho=\rho_1\circ\pi_1\times\cdots\times \rho_d\circ\pi_d.$ Then $\rho$ is faithful on $Z(G)$ and since every non-trivial normal subgroup of a nilpotent group must intersect the center non trivially, $\rho$ is faithful on $G$ as well.
\end{proof}
\subsection{Faithful character case for torsion free nilpotent groups with infinite non-central conjugacy classes}
\begin{lemma} \label{thm:sqdnorm} Let $G$ be a  discrete nilpotent group and let $H\leq G$.  Suppose $\phi$ is a trace on $G$, such that $\phi$ restricted to $H$ is a character.  If $\pi_\phi$ is a QD representation of $G$, then $\pi_{\phi|_H}$ is a QD representation of $H.$ 
\end{lemma}
\begin{proof} By assumption $C^*(G)/\textrm{ker}(\pi_\phi)$ is a QD C*-algebra.  Since $\phi$ is a trace we have
\begin{align}
\textrm{ker}(\pi_\phi)\cap C^*(H)&=\{ x\in C^*(G):\phi(x^*x)=0  \}\cap C^*(H)   \notag \\
&=\{ x\in C^*(H):\phi(x^*x)=0  \}=:J_H \label{eq:traceideal}.
\end{align}
Since $\phi$ is a character of $H$, it follows by Theorem \ref{thm:CM} that $J_H$ is a maximal ideal of $C^*(H)$ and by (\ref{eq:traceideal}) it follows that $C^*(H)/J_H$ embeds into $C^*(G)/\textrm{ker}(\pi_\phi).$  Therefore, $C^*(H)/J_H$ is a QD C*-algebra.  Since $C^*(H)/J_H$ is simple, by Voiculescu's theorem $C^*(H)/J_H$ is strongly QD; in particular $\pi_{(\phi|_H)}$ is a QD representation of $H$.
\end{proof}

\begin{theorem} \label{thm:noncconj} Let $G$ be a finitely generated, torsion free nilpotent group such that every non-central conjugacy class is infinite.  Let $\phi$ be a faithful character on $G.$  Then $\pi_\phi$ is a QD representation.
\end{theorem}

\begin{proof} Since $G$ is finitely generated and nilpotent it is centrally inductive in the sense of \cite{Carey84}, that is any character vanishes on the infinite conjugacy classes of the group modulo the kernel of the representation--see \cite[Proposition 3]{Howe77} and \cite[Section 4]{Carey84}. Therefore $\phi(x)=0$ if $x\not\in Z(G)$ by assumption.
Let $Z(G)\cong \Z^d$ and let $a_1,...,a_d$ be free generators of $Z(G).$ Obtain integers $k_1,...,k_d$ and homomorphisms $\rho_j:G\rightarrow U_{k_j}$ as in Lemma  \ref{lem:correctplace}.   Since $\phi$ is a character, it is multiplicative on $Z(G)$, i.e. there exist real numbers, $\theta_1,...,\theta_d$ such that
\begin{equation*}
\phi\Big( \sum_{j=1}^d \alpha_ja_j  \Big)=\prod_{j=1}^d \exp( i \theta_j\alpha_j).
\end{equation*}
For each $j=1,...,d$ define the character $\psi_j$ on $U_{k_j}$ by $\psi_j(\rho_j(a_j))=\exp(i\theta_j)$ and $\psi_j(x)=0$ if $x\not\in Z(U_{k_j}).$ By Theorem \ref{thm:QDtrace} $\pi_{\psi_j}$ is a QD representation of $U_{k_j}$ for $j=1,...,d.$ Set $U=U_{k_1}\times \cdots\times U_{k_d}$ and define $\psi=\psi_1\otimes \cdots \otimes \psi_d$ on $C^*(U)=C^*(U_{k_1})\otimes \cdots\otimes C^*(U_{k_d}).$ Since the tensor product of QD representations are QD, it follows that $\pi_\psi$ is a QD representation of $U.$  

By Lemma \ref{lem:correctplace}, $\rho=\rho_1\times\cdots\times\rho_d$  is faithful, hence $\psi\circ\rho=\phi.$  Since $\phi$ is a character of $G$, it follows by Lemma \ref{thm:sqdnorm} that $\pi_\phi$ is a QD representation of $G.$
\end{proof}

\section{Reduction to the General Case} \label{sec:RTFC}
 Before we can complete the proof of our main theorem we must show that strong QD is stable under finite extensions. In this section we use the  theory of C*-crossed products and refer the reader to \cite[Section 4.1]{Brown08} for definitions.

\begin{lemma} \label{thm:find} Let $G$ be solvable and $H \leq G$ a normal, virtually nilpotent subgroup of finite index.  If $H$ is strongly QD then so is $G.$
\end{lemma}
\subsection{Proof of Lemma \ref{thm:find}}
For a representation $\rho$ of a group, we will write $\mh_\rho$ for the associated Hilbert space.
\\\\
We proceed by induction on $|G/H|.$  Since $G$ is solvable, so is $L=G/H.$ In particular the abelianization   $L/[L,L]$ is a product of cyclic groups and has order strictly bigger than 1.  Therefore by our induction hypothesis we may assume that $G/H$ is cyclic of prime order, say $p.$
Let $\pi$ be an irreducible representation of $G.$ We will prove that $\pi$ is QD.  It is necessary to split this into two cases, but before we can even define the two cases we must recall a few basic facts.
\\\\
Let $e, x,x^2,...,x^{p-1}\in G$ be coset representatives of $G/H.$ 
Let $\alpha$ denote the action of $G$ on $\ell^\infty(G/H)$ by left translation.  It is well-known, and easy to prove, that $\ell^\infty(G/H)\rtimes_\alpha G\cong M_p\otimes C^*(H)$ and that under this inclusion $C^*(H)\subseteq  C^*(G)\subseteq M_p\otimes C^*(H)$ we may realize this copy of $C^*(H)$ as the C*-algebra generated by the diagonal matrices $(\lambda_h,\lambda_{xhx^{-1}},...,\lambda_{x^{p-1}hx^{-(p-1)}})$ with $h\in H.$
\\\\
By a well-known construction (see \cite[Theorem 5.5.1]{Murphy90}) there is an irreducible representation $id_p\otimes \sigma$ of $M_p\otimes C^*(H),$ such that $\mh_\pi\subseteq \mh_{id_p\otimes \sigma}$ and if $P:\mh_{id_p\otimes \sigma}\rightarrow \mh_\pi$ is the orthogonal projection, then
\begin{equation}
P(id_p\otimes \sigma(x))P=\pi(x)\quad \textrm{ for all }\quad x\in C^*(G). \label{eq:extendrep}
\end{equation}
For $0\leq i\leq p-1$ define the representation $\sigma_i(h)=\sigma(x^ihx^{-i})$ for $h\in H.$
\newline
We now split the proof into two cases.
\subsubsection{Unitarily Equivalent Case} In this subsection we assume that there exists  $0\leq i<j\leq p-1$ and a unitary $u\in B(\mh_\sigma)$ so 
\begin{equation}
\sigma_i(h)=u\sigma_j(h)u^* \textrm{ for }h\in H. \label{eq:ijue}  
\end{equation}
Replacing $h$ with $x^{-i}hx^i$ in (\ref{eq:ijue}) we have 
\begin{equation}
\sigma_0(h)=\sigma(h)=u\sigma(x^{j-i}hx^{-(j-i)})u^*=u\sigma_{j-i}(h)u^* \quad \textrm{for } h\in H. \label{eq:stuff}
\end{equation}
Since $p$ is prime,  $x^{j-i}H$ is a generator for $G/H.$ Let $0\leq k\leq p-1.$  Then there is an $\ell\geq0$ and an $h_1\in H$ so $h_1x^k=x^{\ell(j-i)}.$ 
By repeated application of (\ref{eq:stuff}) we have for $h\in H$
\begin{align*}
\sigma(h)&= u^{\ell}\sigma(x^{\ell(j-i)}hx^{-\ell(j-i)})u^{*\ell}\\
&=u^{\ell}\sigma(h_1)\sigma(x^khx^{-k})\sigma(h_1^{-1})u^{*\ell}\\
\end{align*}
Therefore all of the $\sigma_i$ are unitarily equivalent to each other.  We summarize these observations as
\begin{lemma} \label{lem:ue} Suppose that there are $0\leq i<j\leq p-1$ such that $\sigma_i$ and $\sigma_j$ are unitarily equivalent. Then for any $0\leq k\leq p-1$ there is a unitary $u\in B(\mh_\sigma)$ such that $\sigma=u^*\sigma_ku$ and for any $0\leq \ell \leq p-1$ there is a unitary $v$ in the group generated by $u$  and $\sigma(H)$ such that $\sigma=v^*\sigma_\ell v.$

\end{lemma}

Now let $J\subseteq C^*(H)$ be the kernel of $\sigma.$ Since all of the $\sigma_i$ are unitarily equivalent to each other they all have the same kernel.  It follows that 
\begin{equation}
\textup{For all } y\in G \textup{ we have }\lambda_yJ\lambda_{y^{-1}}=J. \label{eq:Jinv}
\end{equation}
 Let $J_G$ be the ideal of $C^*(G)$ generated by $J.$  By (\ref{eq:Jinv}) we have
\begin{equation*}
J_G=\Big\{ \sum_{i=0}^{p-1} a_i\lambda_{x^i}: a_i\in J \Big\}.
\end{equation*}
We also have
\begin{equation}
J_G\subseteq \textrm{ker}\Big( id_p\otimes\sigma|_{C^*(G)}  \Big)\subseteq \textrm{ker}(\pi). \label{eq:inclusions}
\end{equation}
The first inclusion follows because $J_G$ is contained in the ideal of $M_p\otimes C^*(H)$ generated by $J$ which is precisely the kernel of $id_p\otimes\sigma$ and the second inclusion is due to (\ref{eq:extendrep}).
\\\\
If $J_G$ happens to be a maximal ideal of $C^*(G)$, then we must have $J_G=\textup{ker}(\pi)$ in which case $\textup{ker}(\pi)=\textrm{ker}\Big( id_p\otimes\sigma|_{C^*(G)}  \Big)$ by (\ref{eq:inclusions}).  Then
 $C^*(G)/\textup{ker}(\pi)\subseteq (M_p\otimes C^*(H))/\textup{ker}(id_p\otimes \sigma)$ thus $C^*(G)/\textup{ker}(\pi)$ is QD because $M_p\otimes C^*(H)$ is strongly QD.
\\\\
Consider now the case when $J_G$ is not maximal. We prove the following:
\begin{lemma} \label{lem:twistJ} There is an $a\in C^*(H)\setminus J$ and an index $0<i\leq p-1$ such that $\lambda_h a-a\lambda_{x^i hx^{-i}}\in J$ for all $h\in H.$  
\end{lemma}
\begin{proof} Let $x\in C^*(G).$  It follows from the definition of group C*-algebras that there are unique $a_0,...,a_{p-1}\in C^*(H)$ such that
$x=\sum_{i=0}^{p-1} a_i\lambda_{x^i}.$
In particular, for all $x\in C^*(G)$ the following is well-defined: 
\begin{equation*}
\textup{supp}(x)=\{ i\in \{ 0,...,p-1  \}: a_i\neq0   \}.
\end{equation*}

Let now $z=\sum_{i=0}^{p-1} a_i\lambda_{x^i}\in \textup{ker}(\pi)\setminus J_G$ with $a_i\in C^*(H).$ Since $z\not\in J_G$, there is some index $i$ so $a_i\not\in J.$ By the maximality of $J$ in $C^*(H)$ we have $|\textup{supp}(z)|\geq 2.$  
\\\\
We transform $z$ via the following algorithm so one of its coefficients satisfy the lemma.
\begin{enumerate}
\item[(i)] If $|\textup{supp}(z)|=2$ go to the final step. If not go to the next step.
\item[(ii)] If any of the nonzero $a_i$ are in $J$, replace $z$ with $z-a_i\lambda_{x_i}$. Repeat until we either have $a_i=0$ or $a_i\not\in J$ for all $0\leq i\leq p-1.$ Notice that $|\textup{supp}(z)|$ cannot increase under this operation. If $|\textup{supp}(z)|=2$ go to the last step, if not continue with the next step.
\item[(iii)] If $a_0=0$, replace $z$ with $z\lambda_{x^{-i}}$ where $a_i\neq0$ and we now have $a_0\not\in J.$  Notice that $|\textup{supp}(z)|$ remains constant under this operation.
\item[(iv)] Since $a_0\not\in J$ and $J$ is maximal in $C^*(H)$ there are $\alpha_1,....,\alpha_t,\beta_1,...,\beta_t\in C^*(H)$ and $\gamma\in J$ so $\sum \alpha_i a_0 \beta_i=1+\gamma.$  Replace $z$ with
$\sum \alpha_i z\beta_i-\gamma$ and note that $a_0=1$. Notice that $|\textup{supp}(z)|$ cannot increase under this operation. If $|\textup{supp}(z)|=2$, go to the last step, if not continue with the next step.
\item[(v)] If there is an $i>0$ such that $a_i\in C^*(H)\setminus J$ and  $\lambda_ha_i-a_i\lambda_{x^ihx^{-i}}\in J$ for every $h\in H,$ then terminate the algorithm by taking $a=a_i$ for any such $i.$
\\\\
If not, then there is an $i>0$ with $a_i\in C^*(H)\setminus J$ and an $h\in H$ such that $\lambda_ha_i-a_i\lambda_{x^ihx^{-i}}\not\in J.$ Then replace $z$ with $\lambda_h z-z\lambda_h.$
Since $a_0$ was equal to 1 it follows that $|\textup{supp}(z)|$ has been reduced by at least 1 and $z\in \textup{ker}(\pi)\setminus J_G.$  If $|\textup{supp}(z)|>2$ repeat the algorithm.  If $|\textup{supp}(z)|=2$ go to the next step.
\item[(vi)] After translating and using the maximality of $J$ as in steps (3) and (4) we have  $z=1+a\lambda_{x^i}$ for some $0< i\leq p-1$ and $a\in C^*(H)\setminus J.$  Then for every $h\in H$ we have $(\lambda_h z-z\lambda_h)\lambda_{x^{-i}}=(\lambda_ha-a\lambda_{x^ihx^{-i}})\in \textup{ker}(\pi) \cap C^*(H)= J.$
\end{enumerate}

\end{proof} 
\begin{lemma} \label{lem:case1} Suppose that any two of the representations $\{  \sigma,...,\sigma_{p-1} \}$ are unitarily equivalent. Then $\pi$ is a QD representation.
\end{lemma}
\begin{proof} We first show that for each $0\leq k\leq p-1$ there is a unitary $u\in \sigma(C^*(H))$ so $u^*\sigma u=\sigma_k .$

 By Lemma \ref{lem:ue} is suffices to find \emph{any} $i>0$ and a $u\in C^*(H)$ so $u^*\sigma u=\sigma_i .$  To this end, obtain $a$ and $i$ as in Lemma \ref{lem:twistJ}.  By Lemma \ref{lem:ue} there is a $u\in B(\mh_\sigma)$ so $u^*\sigma u=\sigma_i .$ Then for all $h\in H$ we have 
\begin{equation*}
\sigma(h)\sigma(a)=\sigma(a)\sigma(x^i hx^{-i})=\sigma(a)u^*\sigma(h)u.
\end{equation*}
Then $\sigma(a)u^*\in \sigma(H)'=\mathbb{C}.$  Since  $\sigma(a)\neq 0$ we have $u=\sigma(a)\V \sigma(a)  \V^{-1}\in \sigma(C^*(H)).$ 

We now have unitaries $u_1,...,u_{p-1}\in \sigma(C^*(H))$ such that $u_i^*\sigma u_i=\sigma_i$ for $i=1,...,p-1.$ Define the diagonal unitary
$W=(1,u_1,...,u_{p-1})\in M_p\otimes \sigma(C^*(H)).$  Then
\begin{equation*}
W((\sigma(h),...,\sigma(x^{p-1}hx^{-(p-1)}))W^*=1_p\otimes \sigma(h) \quad \textrm{ for all }h\in H.
\end{equation*} 
Since $W\in M_p\otimes \sigma(C^*(H))$, after conjugating by $W$ we may assume that $(id_p\otimes \sigma)(C^*(G))$ contains a copy of $1_p\otimes \sigma(C^*(H)).$

Let $Q$ be the support projection of $(id_p\otimes \sigma )(\textup{ker}(\pi)).$  Then
\begin{equation*}
Q\in (id_p\otimes \sigma )(C^*(G))'\subseteq [1_p\otimes \sigma(C^*(H))]'=M_p\otimes 1\cong M_p.
\end{equation*}
Since $\textup{ker}(\pi)$ is maximal we have 
\begin{equation*}
C^*(G)/\textup{ker}(\pi)\cong (1-Q)  (id_p\otimes \sigma )(C^*(G))   (1-Q)\subseteq (1-Q)\otimes \sigma(C^*(H)).
\end{equation*}
Since $\sigma(C^*(H))$ is QD, so is $(1-Q)\otimes \sigma(C^*(H)).$  Since subalgebras of QD C*-algebras are QD it follows that  $C^*(G)/\textup{ker}(\pi)$ is also QD.
\end{proof}
\subsubsection{Non unitarily equivalent case}
In this case we assume that none of the $\sigma_i$ are unitarily equivalent.  
\begin{lemma}   If none of the $\sigma_i$ are unitarily equivalent, then $\pi$ is a QD representation.
\end{lemma}
\begin{proof} 
Let $Q$ be the support projection of $(id_p\otimes \sigma )(\textup{ker}(\pi)).$  Since none of the $\sigma_i$ are unitarily equivalent, by \cite[Theorem 3.8.11]{Pedersen79} we have
\begin{equation*}
Q\in (id_p\otimes \sigma )(C^*(G))'\subseteq (id_p\otimes \sigma )(C^*(H))'=\ell^\infty(p)\otimes 1\subseteq M_p\otimes 1.
\end{equation*}
We now complete the proof as in Lemma \ref{lem:case1}.
\end{proof}
\subsection{Proof of Theorem \ref{thm:fgnil}} 
We prove the main theorem via induction on the Hirsch number of a nilpotent group.  Let us recall the relevant results, all of which are classic and can be found in \cite[Chapter 1]{Segal83}. 

A group $G$ is \emph{polycyclic} if  it contains a series $\{ e \}=G_{n+1}\trianglelefteq G_n\trianglelefteq \cdots G_2\trianglelefteq G_1=G$ such that $G_i/G_{i+1}$ is a cyclic group.  The \emph{Hirsch number} of $G$ is the number of times that $G_i/G_{i+1}$ is infinite, and we write $h(G)$ for the Hirsch number.  The Hirsch number is an invariant of a polycyclic group. All finitely generated nilpotent groups are polycyclic and we have for a normal subgroup $N$ of $G$ that
$h(G)=h(N)+h(G/N).$  We shall use these facts without reference in the following
\begin{thm16} Let $G$ be a finitely generated,  solvable virtually nilpotent group.  Then $C^*(G)$ is strongly quasidiagonal.
\end{thm16}
\begin{proof} By Lemma \ref{thm:find} it suffices to prove that $C^*(G)$ is strongly QD for all finitely generated nilpotent groups.

We proceed by induction on $h(G).$ The case $h(G)=0$ is trivial so we assume it holds when $h(G)<d$ and prove it for $h(G)=d.$  It is well-known (see \cite{Hall88}) that $G$ has a finite index, torsion free subgroup (and this doesn't change the Hirsch number). Therefore we may assume $G$ is torsion free by Lemma \ref{thm:find}.  

It follows by a standard argument that $G$ has a finite index subgroup $N$ such that every non-central conjugacy class of $N$ is infinite.  Indeed, let $F(G)$ denote the subgroup of $G$ consisting of those elements with finite conjugacy classes.  Since $G$ is finitely generated and nilpotent, so is $F(G).$  Therefore, by \cite[Lemma 3]{Baer48}, we have that $Z(G)$ has finite index in $F(G).$ Let $x_1,...,x_k$ be coset representatives for $F(G)$ in $Z(G).$  Letting $G$ act by conjugation on the conjugacy classes of the $x_i$ and taking the kernel $N$ of this action, we see that $N$ has finite index in $G$ and  every non-central conjugacy class of $N$ is infinite. Therefore we may again by Lemma \ref{thm:find} assume that every non-central conjugacy class of $G$ is infinite (and again leave the Hirsch number unchanged).

Let now $\phi$ be a character of $G.$ If $\phi$ is faithful, then $\pi_\phi$ is QD by Theorem \ref{thm:noncconj}.  
If $\{ e\}\neq k(\phi)=\{ g\in G:\phi(g)=1 \}$, then $h(k(\phi))\geq1$, since $G$ is torsion free.  Hence $h(G/k(\phi))<d$ and clearly $\phi$ induces a trace, $\widehat{\phi}$ on $G/k(\phi)$ satisfying $\widehat{\phi}\circ q=\phi$ where $q:G\rightarrow G/k(\phi)$ is the quotient map.  By our induction hypothesis,  $\pi_{\widehat{\phi}}$ is a QD representation of $G/k(\phi)$ and by uniqueness of GNS representations, $\pi_{\widehat{\phi}}\circ q$ is unitarily equivalent to $\pi_\phi.$  By Theorem  \ref{thm:CM}, it follows that every primitive quotient of $C^*(G)$ is QD.  By Theorem \ref{thm:pog}, $C^*(G)$ has a $T_1$ primitive ideal space.  Finally $C^*(G)$ is strongly QD by Lemma \ref{lem:Voic}.

\end{proof}

\section*{Note added in proof} 
Regarding Question \ref{ques:fgsqd},  the author has recently shown in \cite{Eckhardt13b} that there are polycyclic (hence finitely generated), non virtually nilpotent, strongly quasidiagonal groups. 

\section*{Acknowledgement}
An earlier version of this paper contained an error in the proof of Lemma 3.1.  We would like to thank Mikael R\o rdam for kindly pointing this out to us.

\end{document}